\documentclass{amsart}

\usepackage{amsthm, amssymb, amscd, amsmath, amsfonts, amsbsy}
\usepackage{stmaryrd}
\usepackage{graphicx}
\usepackage{hyperref}
\usepackage[a4paper, margin=1in]{geometry}
\usepackage{float}
\usepackage{xfrac}
\usepackage{tikz-cd}
\usepackage{mathtools}
\usepackage[percent]{overpic}	
\usepackage{caption, subcaption}

\newcommand{\R}{\ensuremath{\mathbb R}}
\newcommand{\N}{\ensuremath{\mathbb N}}
\newcommand{\Z}{\ensuremath{\mathbb Z}}

\newcommand{\unknot}{\raisebox{-0.22em}{\includegraphics{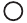}}}
\newcommand\Kauff{(-A^2-A^{-2})}

\newcommand{\qq}[1]{``#1''}

\newcommand{\isoteq}{\stackrel{\text{isot.}}{=}}

\newcommand{\eqeq}[1]{\stackrel{\text{\eqref{#1}}}{=}}

\newcommand{\A}[2]{\raisebox{#1}{\includegraphics[page=#2]{pics/kbsm.pdf}}}

\newcommand{\linksfr}{\mathcal{L}^\text{fr}} 
\newcommand{\linkstop}{\mathcal{L}^\text{top}} 

\newcommand{\submodule}{\mathcal{S}}

\newcommand{\kbsm}{\mathcal{K}}  
\newcommand{\kbsmfr}{\mathcal{K}^\text{fr}} 
\newcommand{\kbsmtop}{\mathcal{K}^\text{top}} 

\newcommand{\basis}{\mathcal{B}}
\newcommand{\basisfr}{\mathcal{B}^\text{fr}}
\newcommand{\basistop}{\mathcal{B}^\text{top}}

\newcommand\Q[1]{\raisebox{-1.5em}{\includegraphics[page=#1,scale=0.75]{pics/diag.pdf}}}
\newcommand{\reid}[1]{\raisebox{0.65em}{\;\;$\xleftrightarrow{\text{#1}}$\;\;}}
\newcommand{\reidf}[1]{\raisebox{1.2em}{\;$\xleftrightarrow{\text{#1}}$\;}}

\newcommand\iconreid[1]{\raisebox{-0.7em}{\includegraphics[page=#1]{icons}}}
\newcommand\icon[1]{ \raisebox{-0.7em}{\includegraphics[page=#1]{icons}}}
\newcommand\nevozel{ \raisebox{-0.7em}{\includegraphics{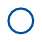}}}

\newcommand\itheta[1]{ \raisebox{-1.85em}{\begin{overpic}[page=45]{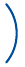}\put(15,61){\small{$#1$}}\end{overpic}}}
\newcommand\iH[1]{ \raisebox{-1.85em}{\begin{overpic}[page=46]{icons}\put(39,44){\small{$#1$}}\end{overpic}}}
\newcommand\iunknot{ \raisebox{-1.85em}{\begin{overpic}[page=48]{icons}\end{overpic}}}
\newcommand\iunknots{ \raisebox{-1.85em}{\begin{overpic}[page=47]{icons}\end{overpic}}}

\theoremstyle{definition}
\newtheorem{definition}{Definition}[section]

\newtheorem{theorem}{Theorem}[section]
\newtheorem{lemma}{Lemma}[section]

\title{The Bracket polynomial of bonded knots and applications to entangles proteins}
\author[B. Gabrov\v sek]{Bo\v{s}tjan Gabrov\v sek}
\address[Bo\v{s}tjan Gabrov\v sek]{
Rudolfovo – Science and Technology Centre Novo mesto, Podbreznik 15, 8000 Novo mesto, Slovenia; University of Ljubljana, Faculty of Education, Kardeljeva ploščad 16, 1000 Ljubljana, Slovenia; Institute of Mathematics, Physics and Mechanics, Jadranska ulica 19, 1000 Ljubljana, Slovenia; University of Ljubljana, Faculty of Mechanical Engineering
        Aškerčeva 6, 1000 Ljubljana, Slovenia.}
        
\email[B.~Gabrov\v sek]{bostjan.gabrovsek@pef.uni-lj.si}

\author[M. Simonič]{Matic Simonič}
\address[Matic Simonič]{
University of Ljubljana, Faculty of Mathematics and Physics, Jadranska ulica 19,
1000 Ljubljana, Slovenia.}
\email[M.~Simonič]{matic.simonic@fmf.uni-lj.si}

\date{\today}


\begin{document}

\begin{abstract}
We model proteins with intramolecular bonds, such as disulfide bridges, using the concept of bonded knots -- closed loops in three-dimensional space equipped with additional bonds that connect different segments of the knot. We extend the Kauffman bracket polynomial (which is closely related to the Jones polynomial) to bonded knots through the introduction of the bonded version of the Kauffman bracket skein module. We show that this module is infinitely generated and torsion-free for both the rigid and topological case of bonded knots, providing an invariant of such structures.
\end{abstract}

\maketitle

\section{Introduction}

In nature, proteins are complex biological structures composed of long chains of amino acids. These chains are stabilized by intramolecular bonds between specific amino acid residues, which play a crucial role in determining the molecule’s shape and enhancing its stability -- both essential for its biological function \cite{dabrowski2016search, zhao2018stability, sulkowska2008stabilizing}, see \autoref{fig:mamba}. Mathematically, such proteins can be modeled as either open or closed knots with connections between distinct regions of the knot. Such configurations are referred to as \emph{bonded knots}. In this paper we introduce two topological invariants that are capable of distinguishing between different topological types of bonded knots, providing a deeper understanding of their structural and functional roles in biological systems.

\begin{figure}[h]
    \centering
    \includegraphics[width=0.6\textwidth, page=1]{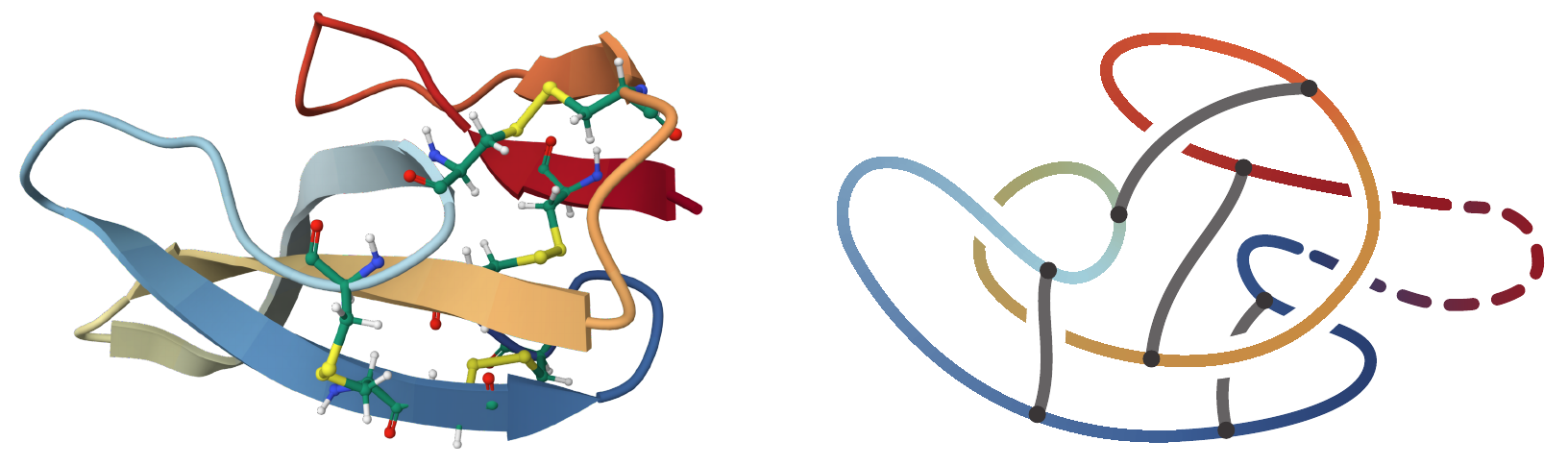} 
    \caption{{\bf The bonded structure of a protein.} Right: the 3D ribbon model of the toxin FS2 isolated from black mamba (Dendroaspis polylepis) venom (PDB 1TFS), left: the associated bonded knot diagram closed with a direct segment. The PDB code refers to the protein structure entry in the Protein Data Bank \cite{burley2019rcsb}.} 
    \label{fig:mamba} 
\end{figure}

From a topological point of view, proteins have so far been extensively studied in numerous settings, such as knots or links \cite{mansfield1994there, virnau2006intricate, dabrowski2017tie, millett2013identifying, dabrowski2019knotprot, wikoff2000topologically, gabrovvsek2023invariants}, slipknots \cite{sulkowska2012conservation, dabrowski2016lassoprot}, cystine knots \cite{simien2021topological, cardoso2019structure, dabrowski2019knotprot} and lassos \cite{dabrowski2016lassoprot, niemyska2016complex} (see \autoref{fig:structures}).
 More recently, seven distinct topological types of $\Theta$-curves, a bonded knot with one bond, were identified in proteins (\autoref{fig:structures}, middle image) \cite{dabrowski2024theta, bruno2024knots}. The study also explored both the folding pathways of these structures and the stabilizing effects provided by the bonds, shedding light on how bonds contribute to protein stability and function.

 To formally analyze bonded topological structures, we introduce in this paper the mathematical theory behind their classification, the normalized bonded bracket polynomial, an invariant derived from the theory of skein modules independently introduced by Turaev  and Przytycki \cite{Turaev1990, Przytycki1991}. 
 Originally developed to extend polynomial knot invariants such as the Jones and HOMFLYPT polynomials to knots and links in arbitrary 3-manifolds, we adopt the theory of skein modules to develop a well-defined invariant for bonded knots.

 Bonded structures, such as bonded knots and bonded knotoids, have been the subject of recent studies, such as \cite{Goundaroulis2017, adams2020knot, gabrovvsek2021invariant, gugumcu2022invariants}, the importance of applying these models to biological contexts is growing and has been investigated in \cite{sulkowska2020folding, dabrowski2024theta}. However, the theory remains complex and has not yet fully addressed the practical requirements posed by biological research, indicating a need for further theoretical development, a gap which we address in this paper.

\begin{figure}[h]\label{fig:str}
    \centering
    \includegraphics[width=0.85\textwidth,page=2]{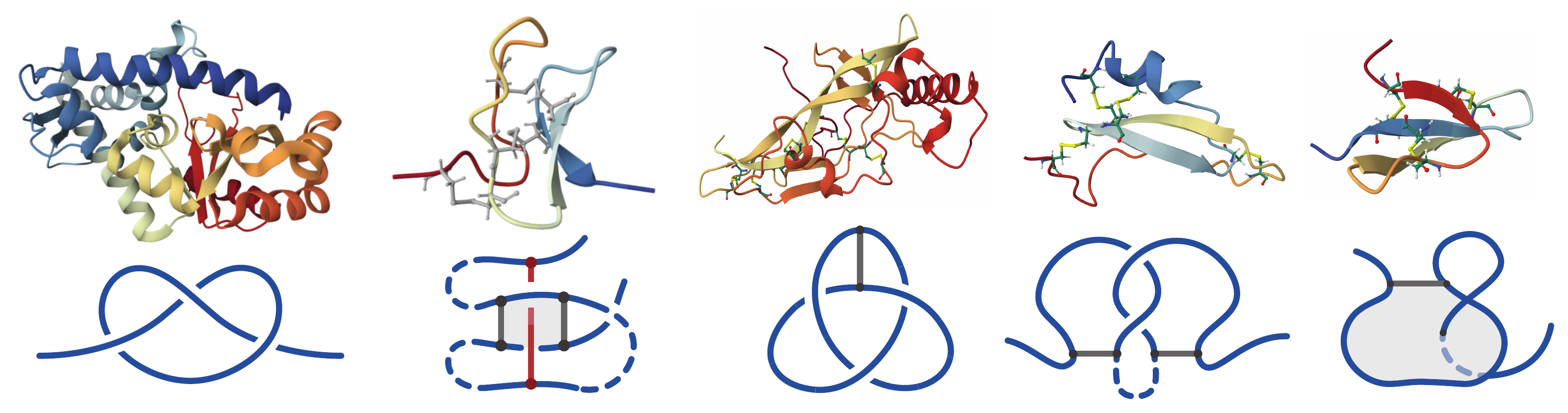} 
    \caption{{\bf Examples of entangles structures in proteins.} From left to right: \emph{an open knot} (PDB 7ECD), \emph{a cystine knot} (PDB 2MXM), where a bond threads through a loop created by the other two bonds, \emph{a deterministic $\theta$-curve} (PDB  1AOC), \emph{a deterministic link} (PDB 2LFK) and \emph{a lasso}, where a tail pierces the minimal surface enclosed with a bond (PDB 1EWS). The last four structures are bonded structures (the protain contains disulfide bridges). Some examples are taken from \cite{sulkowska2020folding}.}
    \label{fig:structures} 
\end{figure}


 The paper is structured as follows. In \autoref{sec:defs} we introduce rigid-vertex, framed and topological bonded links, in \autoref{sec:kbsm} we define the framed bonded Kauffman bracket skein module, show that the module is freely generated by a set of $\Theta$-curves and handcuff links, and introduce the framed bonded bracket polynomial. In section \autoref{sec:topological} we discuss the topological setting and in \autoref{sec:example}
we compute explicit examples. 

\section{Definitions}\label{sec:defs}

A \emph{bonded link} is a pair $(L, \mathbf{b})$, where:
\begin{itemize}
\item $L$ is an oriented link embedded in the 3-sphere $S^3$ (or $\R^3$), that is, a union of a finite number of knots that do not intersect each other:
$$L: S^1 \sqcup S^1 \sqcup \cdots \sqcup S^1 \;\hookrightarrow\; \R^3.$$
\item  $\mathbf{b} = \{b_1, b_2,\ldots,b_n\},n\in \N$, is a set of pairwise disjoint \emph{bonds} (closed intervals) embedded into $S^3$ in such a way, that only the endpoints of $b_i,$ lie in $L$ for all $i\in \{1,2,...,n\}$, i.e. $\partial b_i \subset L$  for all $i\in \{1,2,...,n\}$
(see right-hand side of \autoref{fig:mamba}).
\end{itemize}

The endpoints of bonds form vertices with the link, so bonded links are a special cases of edge-colored trivalent spatial graphs.

A \emph{bonded link diagram} is the regular projection of a bonded link to a plane in such a way that the only type of multiple points are double points intersecting transversely. We call such double points \emph{crossings} of a diagram. A diagram also holds information about over- and under-crossings (see right-hand side of \autoref{fig:mamba}).

Two bonded knots $B_1$ and $B_2$ are isotopic, if there exist an ambient isotopy $H: S^3 \times [0,1] \rightarrow S^3$ that takes $B_1$ into $B_2$, i.e. $H_0 = \text{id}_{S^3}$ and $H_1(B_1) = B_2$. Ambient isotopy is generated by Reidemeister moves as stated in the following theorem.

\begin{theorem}[\cite{Kauffman1989, gabrovvsek2021invariant}]
Two bonded links are ambient isotopic if and only if their diagrams are related by a finite sequence of (Reidemeister) moves I - V depicted in \autoref{fig:reid}.
\end{theorem}
\begin{figure}[ht]
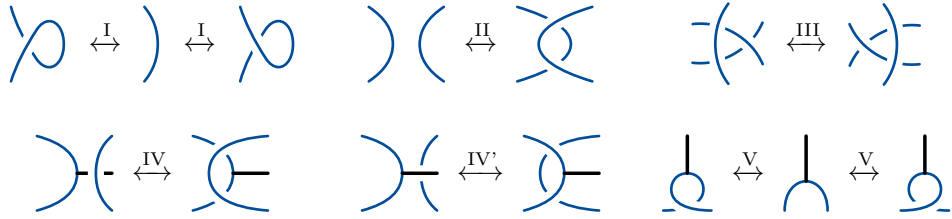

\begin{tabular}{ccccc}
$\iconreid{49}\reid{I}\iconreid{1}\reid{I}\iconreid{2}$ &\quad&
$\iconreid{3}\reid{II}\iconreid{4}$ &\quad&
$\iconreid{5}\reid{III}\iconreid{6}$ \\[2em]
$\iconreid{7}\reid{IV}\iconreid{8}$ &&
$\iconreid{9}\reid{IV'}\iconreid{54}$ &&
$\iconreid{55} \reid{V}\iconreid{56} \reid{V}\iconreid{57}$
\end{tabular}
  \caption{Reidemesiter moves for bonded links. Although not indicated in the figures, arcs in moves I, II, III and the free arcs in IV and IV' can be either link arcs or (colored) bonds.} \label{fig:reidemeister}
  \label{fig:reid}
\end{figure}

In this setting, we say that bonded links contain \emph{topological vertices}, however, if we replace the move V with the move RV in \autoref{fig:rigid}, we obtain so called bonded links with \emph{rigid vertices}. The physical interpretation of a rigid vertex is a disk with fixed strands emanating from its boundary.

\begin{figure}[ht]
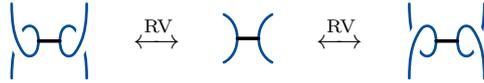

$$ \iconreid{14}\quad\reid{RV}\quad\rotatebox{90}{\iconreid{13}}\quad\reid{RV}\quad\iconreid{15}$$
  \caption{The move RV (a rigid-vertex versions of the move V).} \label{fig:rigid}
\end{figure}
\begin{definition}
Two rigid-vertex bonded links are isotopic if their diagrams are related by a finite sequence of (Reidemeister) moves I -- IV and RV.
\end{definition}

In literature also the following move can be found in the set of Reidemeister moves for bonded structures:
\begin{equation}\label{eq:7}
\raisebox{-1.0em}{\includegraphics[scale=0.8,page=26]{pics/kbsm.pdf}}
\;\leftrightarrow\;
\raisebox{-1.0em}{\includegraphics[scale=0.8,page=20]{pics/kbsm.pdf}}
\;\leftrightarrow\;
\raisebox{-1.0em}{\includegraphics[scale=0.8,page=21]{pics/kbsm.pdf}}
\end{equation}
We show here that this move is in fact redundant, since it can be expressed as a sequence of moves I--RV:
\begin{align*}
\raisebox{-2.5em}{\includegraphics[scale=0.8,page=27]{pics/kbsm.pdf}}
&\xrightarrow{\text{IV, II}}
\raisebox{-2.5em}{\includegraphics[scale=0.8,page=22]{pics/kbsm.pdf}}
\xrightarrow{\text{IV, I}}
\raisebox{-2.5em}{\includegraphics[scale=0.8,page=23]{pics/kbsm.pdf}}\\[-0.5em]
&\xrightarrow{\text{IV, II}}
\raisebox{-2.5em}{\includegraphics[scale=0.8,page=24]{pics/kbsm.pdf}}
\xrightarrow{\text{IV, I}}
\raisebox{-2.5em}{\includegraphics[scale=0.8,page=25]{pics/kbsm.pdf}}
\xrightarrow{\text{RV}}
\raisebox{-2.5em}{\includegraphics[scale=0.8,page=28]{pics/kbsm.pdf}}.
\end{align*}
A \emph{framed link} link is the embedding of a finite number of ribbons (annuli) into $S^3$, 
$$L^\text{fr}: (S^1\!\times I) \;\sqcup\; (S^1\!\times I) \;\sqcup\; \cdots \;\sqcup\; (S^1\!\times I) \; \hookrightarrow\; S^3.$$ 
A \emph{framed bonded link} is obtained by adding $n$ pairwise disjoint bonds, which are strips homeomorphic to $I\times I$ into $S^3$ in such a way that $\partial I \times I$ lie on $S^1 \times \partial I$. This local glueing looks as the \qq{T} shape depicted on the left-hand side of Figure~\ref{fig:T}. Framed bonded links are special cases of framed spatial graphs studied in \cite{Bao20}.


\begin{figure}[hbt!]
    \includegraphics[scale=1,page=1]{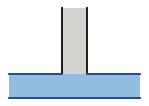}\qquad\qquad
    \includegraphics[scale=1,page=2]{pics/fr.pdf}\qquad
    \includegraphics[scale=1,page=3]{pics/fr.pdf}
\caption{Left: a rigid vertex of a framed bonded link, right: framing markers.}\label{fig:T}
\end{figure}

A \emph{diagram of a framed bonded link} $L$ is the diagram of underlying bonded link (the deformation retract of $L$) together with markers $\raisebox{-0.1em}{\includegraphics[page=83]{pics/proofs.pdf}}$ or $\raisebox{-0.1em}{\includegraphics[page=82]{pics/proofs.pdf}}$ representing positive and negative half twists relative to the blackboard framing, respectively, as depicted on the left-hand side of \autoref{fig:T}.

Two framed bonded links are equivalent if they are ambient isotopic. The following theorem connects equivalence in terms of framed bonded diagrams.

\begin{theorem}[\cite{viro2007quantum, Bao20}]
Two framed bonded links are equivalent if and only if they are connected through a finite sequence of (Reidemeister) moves F0 -- FV depicted in \autoref{fig:frreid}.
\end{theorem}

\begin{figure}[ht]
$$\includegraphics[page=129]{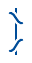}\reidf{F0}\includegraphics[page=58]{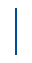}\reidf{F0}\includegraphics[page=59]{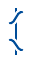} 
\hspace{4em}
\includegraphics[page=60]{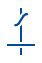}\reidf{$\text{F0}'$}\includegraphics[page=61]{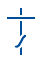}
\hspace{4em}
\includegraphics[page=63]{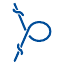} \reidf{FI}\includegraphics[page=58]{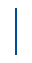} \reidf{FI} \includegraphics[page=64]{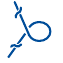}
$$
\vspace{0.0em}
$$\includegraphics[page=66]{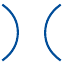}\;\; \reidf{FII} \includegraphics[page=67]{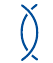}
\hspace{6em}
\includegraphics[page=68]{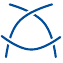} \;\;\reidf{FIII}\;\; \includegraphics[page=69]{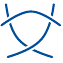}
$$
\vspace{0.0em}
$$
\includegraphics[page=72]{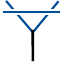} \reidf{FIV} \includegraphics[page=73]{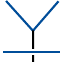} 
\hspace{2em} 
\includegraphics[page=70]{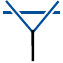} \reidf{$\text{FIV}'$} \includegraphics[page=71]{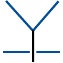}  
\hspace{2em}
\includegraphics[page=111]{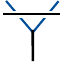} \reidf{$\text{FIV}''$} \includegraphics[page=112]{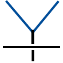} 
\hspace{2em}
\includegraphics[page=113]{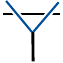} \reidf{$\text{FIV}'''$} \includegraphics[page=114]{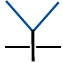}
$$
\vspace{0.0em}
$$
\includegraphics[page=75]{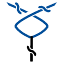}  \reidf{FV} \includegraphics[page=74]{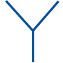} \reidf{FV} \includegraphics[page=76]{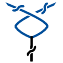}
$$ 
  \caption{Reidemeister moves for bonded framed links. Although not indicated in the figures, arcs in moves F0, FI, FII, FIII can be either link arcs or (colored) bonds.} \label{fig:frreid}
\end{figure}

If we want to keep the framing of the bond intact, we can replace the move FV with the move FRV depicted in \autoref{fig:framedRV}. 

\begin{figure}[ht]
$$
\raisebox{-1.2em}{\includegraphics[page=85]{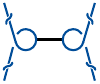}}
\;\raisebox{0em}{\reid{RV}}\;
\raisebox{-0.5em}{\includegraphics[page=56]{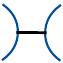}}
\;\raisebox{0em}{\reid{RV}}\;
\raisebox{-1.0em}{\includegraphics[page=86]{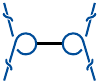}}
$$
  \caption{The move RV is composed of two V moves from \autoref{fig:frreid} performed consecutively in such a way that framing on the bond remains intact.} \label{fig:framedRV}
\end{figure}

By convention, we consider only framed bonded knot diagrams without twists on the bonds (i.e., bonds have blackboard framing), as any twist can always be removed from the bonds using the FV move.

\section{The bonded Kauffman bracket skein module}\label{sec:kbsm}

Before defining the Kauffman bracket skein module (KBSM), let us first recall the definition of the classical (Kauffman) bracket polynomial for framed links in $S^3$ \cite{kauffman1987state}.
\begin{definition}\label{def:bracket}
    Let $L$ be a diagram of a framed link in $S^3$. The \emph{bracket polynomial} $\langle L \rangle$ is a Laurent polynomial in variable $A$ defined by the following three rules:
    \begin{equation} \label{eq:skein123}
    \begin{tabular}{lllr}
         1. && $\Big\langle \nevozel \Big\rangle = 1$ & \\[0.8em]
         2. && $\Big\langle\icon{10}\Big\rangle = A \Big\langle\icon{11}\Big\rangle + A^{-1} \Big\langle\icon{12}\Big\rangle$ & (skein relation) \\[0.8em]
         3. && $\Big\langle \nevozel \sqcup L \Big\rangle = (-A^2-A^{-2})\,L$ &  (framing relation)
    \end{tabular}
    \end{equation}
\end{definition}

To define the bracket polynomial for unframed links, we must normalize it by orienting the diagram and multiplying the polynomial by $(-A^3)^{-w(L)}$, where $w(L)$ is the writhe of the diagram $L$ defined as the sum of crossing signs of the diagram from \autoref{fig:sign}. The normalized bracket polynomial is thus 
$$X(L) = (-A^3)^{-w(L)} \langle L \rangle.$$
By substituting $A=-t^{-1/4}$, we obtain the Jones polynomial $V(L) = X(L)(t^{-1/4})$ \cite{kauffman1987state, mroczkowski2022infinitely}.

 \begin{figure}[ht]
 \begin{tabular}{ccc}
     \includegraphics[page=125]{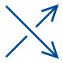} & \qquad\qquad &  \includegraphics[page=126]{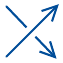} \\
     $+1$ & & $-1$ 
 \end{tabular}
   \caption{The sign of a crossing of an oriented diagram is determined by the right-hand rule.} \label{fig:sign}
 \end{figure}

Using the skein and framing relation \eqref{eq:skein123}, we can simplify two consecutive markers, or equivalently, a positive and negative \emph{kink} (twist) in the following way:
\begin{equation}\label{eq:kinkA}
    \raisebox{-1.2em}{\includegraphics[page=127]{pics/proofs.pdf}} =
    \raisebox{-1.2em}{\includegraphics[page=123]{pics/proofs.pdf}} 
    = -A^3 \raisebox{-1.2em}{\includegraphics[page=124]{pics/proofs.pdf}} \qquad \text{and} \qquad
        \raisebox{-1.2em}{\includegraphics[page=128]{pics/proofs.pdf}} =
        \raisebox{-1.2em}{\includegraphics[page=122]{pics/proofs.pdf}} 
    = -A^3 \raisebox{-1.2em}{\includegraphics[page=124]{pics/proofs.pdf}}\!\!\!\!\!\!\!.
\end{equation}

Turaev and Przytycki \cite{Turaev1990, Przytycki1991} generalized the bracket polynomial to the Kauffman bracket skein module for links in an arbitrary 3-manifold $M$. However, the structure of this module depends on  $M$; it can be infinitely generated and may contain torsion \cite{gabrovvsek2018kbsm, gabrovvsek2017tabulation, gabrovvsek2014homflypt, gabrovvsek2013categorification}.

We now generalize the bracket polynomial, using the algebraic setting of \cite{Turaev1990, Przytycki1991}, to define it for framed bonded links. 

\begin{definition}
Let $R$ be an integral domain, where $A$, $1+A^4$, and $1+A^4+A^8$ are invertible in $R$.
Let $\linksfr$ be a set of rigid-vertex framed bonded links and let $R\linksfr$ be the free $R$-module spanned by $\linksfr$. Let $\submodule$ be the submodule of $R\linksfr$ generated by skein expressions  (cf.~\eqref{eq:skein123})
\begin{equation}\label{eq:skein}
\icon{10} - A \icon{11} - A^{-1} \icon{12} \qquad \text{and} \qquad L \sqcup \nevozel - \Kauff L.
\end{equation}
We define the framed bonded Kauffman bracket skein module as the quotient module $$\kbsmfr = R\linksfr \,/\,\submodule.$$
\end{definition}



After removing all crossings and trivial components via (\ref{eq:skein}), the remaining expression consists of $\Theta$-curves and handcuff links. Let $\Theta$ and $H$ be the following \emph{elementary framed bonded knots} with zero blackboard framing:
\begin{equation}\label{eq:elementary}
\Theta = \raisebox{-2.1em}{\begin{overpic}[page=22]{icons}
\end{overpic}}
\quad\qquad  \text{and} \quad\qquad
H = \;\raisebox{-2.1em}{\begin{overpic}[page=21]{icons}.
\end{overpic}}
\end{equation}
Let $\basisfr$ be the set of all finite unordered products of $\Theta$ and $H$,

$$ \basisfr 
= \Big \{
 \Theta^{m} H^{n} \mid  m, n \in \mathbb N_0
\Big\},
$$
where $\Theta^{0} H^{0} = \emptyset$. Here the multiplication operation is the disjoint sum and is clearly commutative.

\begin{theorem}\label{thm:main}
The module $\kbsmfr$ is freely generated by $\basisfr$.
\end{theorem}

In  order to prove the theorem, we need to show that $\basisfr$ is the generating set of $\kbsmfr$, which we show in \autoref{sec:gen} and that $\basisfr$ is linearly independent, which we show in \autoref{sec:free}.

We will prove the theorem for any number of bonds, $n \in \mathbb{N}$, so let us introduce a grading on the module. Let $\linksfr_n$ be the set of framed links with $n$ bonds. The set of framed links is
$$\linksfr = \bigcup_{n=0}^{\infty} \linksfr_n.$$
This grading induces a natural grading on the skein module:
$$
\kbsmfr = \bigoplus_{n=0}^{\infty} \kbsmfr_n,
$$
where $\kbsmfr_n$ consists of classes of framed links with $n$ bonds.
For convenience, we also introduce a grading on the the generating set of $\kbsmfr_n$:
\begin{align*}
\basisfr_n
&= \Big \{
 \Theta^{i} H^{n-i} \mid  i \in \{0,1,2,...,n\} 
\Big\} 
  \\
&= \Big \{ H^{n},\Theta^{1} H^{n-1}, \Theta^{2} H^{n-2},\cdots,\Theta^{n-2} H^{2}, \Theta^{n-1} H^{1}, \Theta^{n} \Big\}  
\end{align*}
Clearly,
$$\basisfr = \bigcup_{n=0}^{\infty} \basisfr_n.$$




\begin{lemma}\label{lema}
Let $\Theta \,\#\,L$ be a connected sum of $\Theta$ and a bonded link $L$ and let $H \,\#\, L$ be a connected sum of $H$ and $L$. In $\kbsmfr$ it holds
\begin{equation} \label{eq:sum}
\Theta \, \# \, L = \frac{1}{-A^2-A^{-2}}\,\Theta\, L  \qquad \text{and} \qquad
    H \, \# \, L = \frac{1}{-A^2-A^{-2}}\,H\, L.
\end{equation}
\end{lemma}
\begin{proof}
    $$
    -A^3 \, 
    \raisebox{-0.9em}{\includegraphics[page=60]{pics/diag.pdf}} \overset{RV, I, \eqref{eq:kinkA}}{=}
    \raisebox{-0.9em}{\includegraphics[page=61]{pics/diag.pdf}} \overset{\eqref{eq:skein}}{=}
    A \, \raisebox{-0.9em}{\includegraphics[page=62]{pics/diag.pdf}} + 
    A^{-1}\,\raisebox{-0.9em}{\includegraphics[page=60]{pics/diag.pdf}} 
    $$
yields    $$
    \raisebox{-0.9em}{\includegraphics[page=60]{pics/diag.pdf}} 
    = \frac{1}{-A^2-A^{-2}}\,\raisebox{-0.9em}{\includegraphics[page=62]{pics/diag.pdf}}.
    $$
    By an analogous argument
    $$
    \raisebox{-0.9em}{\includegraphics[page=63]{pics/diag.pdf}} 
    = \frac{1}{-A^2-A^{-2}}\,\raisebox{-0.9em}{\includegraphics[page=64]{pics/diag.pdf}}. 
    $$
\end{proof}

\subsection{The generating set}\label{sec:gen}



First, we show that for a framed bonded knot, a bond can be separated using the relators \eqref{eq:skein}. To achieve this, we first isolate the bond by pushing the intersecting arcs aside using moves IV and IV’. Next, we reposition the bond to one side of the diagram using Reidemeister moves, as illustrated in the first diagram of \eqref{eq:anti1}. 

We push the bottom arc over the remaining arcs and compute:
\begin{equation} \label{eq:anti1}
\begin{split}
\Q{1} & \isoteq -A^3 \Q{2}  =  -A^3 \Bigg( A\Q{3} + A^{-1}\Q{4}  \Bigg) \\
& = -A^4 \Bigg( A\Q{5} + A^{-1}\Q{6}  \Bigg) - A^2 \Q{9}\\
& =  -A^5 \Bigg( A \Q{7} + A^{-1}\Q{8}  \Bigg) - A^3 \Bigg(  A\Q{54} + A^{-1}\Q{55}  \Bigg)- A^2 \Q{9}\\
& = -A^6 \Q{10} - A^4 \Q{13} - A^4 \Q{23} - A^2 \Q{11} - A^2 \Q{9}.
\end{split}
\end{equation}
Similarly, we push the bottom arc under the rest of the arcs and compute:
\begin{equation} \label{eq:anti2}
\begin{split}
\Q{1} & \isoteq -A^{-3} \Q{15}  =  -A^{-3} \Bigg( A\Q{17} + A^{-1}\Q{16}  \Bigg) \\
& = -A^{-2} \Q{9}   -A^{-4}\Bigg( A\Q{19} + A^{-1}\Q{18}  \Bigg) \\
& = -A^{-2} \Q{9}   -A^{-3}\Bigg( A\Q{55} + A^{-1}\Q{54}  \Bigg) -A^{-5}\Bigg( A\Q{8} + A^{-1}\Q{7}  \Bigg) \\
& = -A^{-2} \Q{9}   -A^{-2}\Q{11} -A^{-4} \Q{23} - A^{-4} \Q{13} - A^{-6} \Q{10}.
\end{split}
\end{equation}
We solve the equation for the non-composite terms and obtain the solution
\begin{equation*} 
\begin{split}
(1+A^4+A^8)\Q{1} & =- A^4  \Q{13}  - A^4 \Q{23} - (A^6+A^2) \Q{11}  - (A^6 + A^2) \Q{9} \\
(1+A^4+A^8)\Q{10} & =- (A^6+A^2) \Q{13}  - (A^6+A^2) \Q{23} - A^4 \Q{11}  - A^4 \Q{9}.
\end{split}
\end{equation*}
Applying Lemma \ref{lema} on the right-hand side terms and multiplying by $-A^{-2}$, we obtain
\begin{equation} \label{eq:anti4}
\begin{split}
(1+A^4)(1+A^4+A^8)\Q{1} & = A^6 \Theta \Q{21}  + A^6 H \Q{20} + (A^4+A^8) H \Q{21} + (A^4+ A^8) \Theta \Q{20} \\
(1+A^4)(1+A^4+A^8)\Q{10} & = (A^4+A^8)\Theta  \Q{21}  + (A^4+A^8) H\Q{20} + A^6 H \Q{21} + A^6 \Theta \Q{20}
\end{split}
\end{equation}

By repeating this process, we can express a diagram with $n$ bonds as a sum of terms consisting of trivial $\Theta$-curves and handcuff links. In other words, the set $\basis_n$ is a generating set of $\kbsm_n$.

\subsection{Freeness}\label{sec:free}


To prove that $\kbsmfr_n$ is freely generated by $\basisfr_n$ for every $n\in\N$, we first introduce the following definition and lemma (cf.~\cite{Kauffman1989}).

\begin{definition}
   Let $n\in \N,\;i\in \{1,2,\ldots,n\}$, and let $g^{b_i}_0, g^{b_i}_\infty : \linksfr_{n} \rightarrow \linksfr_{n-1}$ be maps which remove the bond $b_i$ using the following conventions:
$$
g^{b_i}_0 : \icon{13} \mapsto \icon{12} \qquad \text{and} \qquad g^{b_i}_\infty : \icon{13} \mapsto \icon{11}.
$$
\end{definition}
\begin{lemma}
   Maps $g^{b_i}_0, g^{b_i}_\infty : \linksfr_{n} \rightarrow \linksfr_{n-1}$ defined as above are well deined in the sense that if bonded links $L, L' \in \linksfr_n$ are isotopic ($L \sim L'$), it follows that $g^{b_i}_0(L) \sim g^{b_i}_0(L')$ and $g^{b_i}_{\infty}(L) \sim g^{b_i}_{\infty}(L')$.
\end{lemma}

 \begin{proof}
Let $L, L'\in \linksfr_{n}$ and $b_i,\; i\in\{1,2,\ldots,n\}$ a bond. We can verify well-definedness for each Reidemeister move separately. Since moves 0, I, II and III do not contain any bonds it is trivial that $g^{b_i}_0(L) \sim g^{b_i}_0(L')$ and $g^{b_i}_{\infty}(L) \sim g^{b_i}_{\infty}(L')$, where $L'$ and $L$ differ by moves 0, I, II and III. \\
Since $g^{b_i}_0$ and $g^{b_i}_{\infty}$ remove the bond $b_i$, while preserving the rest of the bonded link, it suffices to analyze the local structure around the bond for the rest of the moves, IV and V.
We first check that the maps $g^{b_i}_0$ and $g^{b_i}_\infty$ are well-defined under the IV-moves; we only check the move IV'', since all other IV moves follow the same pattern. Applying $g^{b_i}_0$, we have:
$$
g^{b_i}_0\bigg(
\raisebox{-2.1em}{\includegraphics[page=115]{pics/proofs.pdf}}
\bigg) = 
\raisebox{-2.1em}{\includegraphics[page=116]{pics/proofs.pdf}}
\sim 
\raisebox{-2.1em}{\includegraphics[page=119]{pics/proofs.pdf}} =
g^{b_i}_0\bigg(
\raisebox{-2.1em}{\includegraphics[page=117]{pics/proofs.pdf}}
\bigg),
$$

$$
g^{b_i}_0\bigg(
\raisebox{-1.3em}{\includegraphics[page=111]{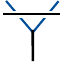}}
\bigg)
= 
\raisebox{-1.3em}{\includegraphics[page=91]{pics/proofs.pdf}} 
\sim 
\raisebox{-1.3em}{\includegraphics[page=92]{pics/proofs.pdf}}
=g^{b_i}_0\bigg(
\raisebox{-1.3em}{\includegraphics[page=112]{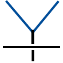}}
\bigg),
$$

\medskip

\noindent and applying $g^{b_i}_\infty$, we have:
$$
g^{b_i}_\infty\bigg(
\raisebox{-2.1em}{\includegraphics[page=115]{pics/proofs.pdf}}
\bigg) 
= 
\raisebox{-2.1em}{\includegraphics[page=120]{pics/proofs.pdf}} 
\sim 
\raisebox{-2.1em}{\includegraphics[page=121]{pics/proofs.pdf}}
=
g^{b_i}_\infty\bigg(
\raisebox{-2.1em}{\includegraphics[page=117]{pics/proofs.pdf}}
\bigg).
$$
$$
g^{b_i}_\infty\bigg(
\raisebox{-1.3em}{\includegraphics[page=111]{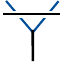}}
\bigg) 
= 
\;\raisebox{-1.3em}{\includegraphics[page=102]{pics/proofs.pdf}} \;
\sim 
\;\raisebox{-1.3em}{\includegraphics[page=103]{pics/proofs.pdf}}\;
=
g^{b_i}_\infty\bigg(
\raisebox{-1.3em}{\includegraphics[page=112]{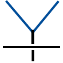}}
\bigg).
$$


\medskip

\noindent Next, we check that the $g^{b_i}_0$ is well-defined under the move RV:
$$
g^{b_i}_0\bigg(
\raisebox{-1.15em}{\includegraphics[page=98]{pics/proofs.pdf}}
\bigg) 
= 
\raisebox{-1.15em}{\includegraphics[page=99]{pics/proofs.pdf}} 
\raisebox{-0.6em}{\reid{I,I}}
\raisebox{-1.6em}{\includegraphics[page=100]{pics/proofs.pdf}}
=
g^{b_i}_0\bigg(
\raisebox{-1.6em}{\includegraphics[page=101]{pics/proofs.pdf}}
\bigg),
$$
Applying $g^{b_i}_\infty$, we have:
$$
g^{b_i}_\infty\bigg(
\raisebox{-1.15em}{\includegraphics[page=98]{pics/proofs.pdf}}
\bigg) 
= 
\raisebox{-1.15em}{\includegraphics[page=108]{pics/proofs.pdf}}
\raisebox{-0.6em}{\reid{O,O,II}}
\raisebox{-1.6em}{\includegraphics[page=110]{pics/proofs.pdf}}
=
g^{b_i}_\infty\bigg(
\raisebox{-1.6em}{\includegraphics[page=101]{pics/proofs.pdf}}
\bigg),
$$
Proofs for the mirror versions are analogous.
 \end{proof}


The two maps $g^{b_i}_0$ and $g^{b_i}_\infty$ induce maps $(g^{b_i}_0)^*$ and $(g^{b_i}_\infty)^*$ on the skein module, respectively.
The induced maps map the generators in the following way:

 \begin{center}
 \begin{tabular}{ccc}
   \begin{tabular}{lcll}
$(g^{b_i}_0)^*:$ & \raisebox{0.8em}{\itheta{b_i}} & $\mapsto$ & \!\raisebox{0.8em}{\iunknot} \\
& \raisebox{0.8em}{\iH{b_i}} & $\mapsto$ & $\raisebox{0.8em}{\iunknots} =\Kauff \raisebox{0.8em}{\iunknot}$ \\
\end{tabular}
   & \quad &
   \begin{tabular}{lcll}
$(g^{b_i}_\infty)^*:$ & \raisebox{0.8em}{\itheta{b_i}} & $\mapsto$ & $\raisebox{0.8em}{\iunknots} = \Kauff \raisebox{0.8em}{\iunknot}$ \\
& \raisebox{0.8em}{\iH{b_i}} & $\mapsto$ & \raisebox{0.8em}{\iunknot} \\
\end{tabular}
 \end{tabular}

 \end{center}

\medskip

\noindent We will now show that $\kbsm_n$ is freely generated by $\basis_n$ by induction on $n \in \N$.

\bigskip

\noindent (Base of induction.) First, let us show linear independence of $\basis_1 =\{\Theta, H\}$ in $\kbsm_1$, i.e.
\begin{equation}\label{eq:smallsum}
     p\Theta+qH=0\Longrightarrow q=p=0,    
\end{equation}
    where $p,q\in R$. Let $b_1\in \mathbf{b}$ (elements in $\basis_1$ have only one bond) and let us apply $(g_0^{b_1})^*$ to the sum \eqref{eq:smallsum}:
       $$
       p\unknot+q(-A^2-A^{-2})\unknot=0,
       $$
       where $\unknot$ is the unknot. 
       Applying $(g_{\infty}^{b_1})^*$ to the sum \eqref{eq:smallsum}, we get: 
       $$
       p(-A^2-A^{-2})\unknot+q\unknot=0.
       $$
       We solve this system of equations and obtain
       $$
       q((-A^2-A^{-2})^2-1)\unknot =0, \qquad \text{and}\qquad
       p((-A^2-A^{-2})^2-1)\unknot =0.
       $$
       From the fact that $R$ is an integral domain and $(-A^2-A^{-2})^2-1 \neq 0$, and the fact that the Kauffman bracket skein module of links in $S^3$ (corresponding to the classical bracket polynomial) is torsion-free, it follows that $p=q=0$.
\bigskip

 \noindent (Induction Step.) We now assume that $\kbsm_{n-1}$ is freely generated by $\basis_{n-1}$ and we will show that it follows that $\kbsm_n$ is freely generated by $\basis_n$. 
Given $\basis_n= \{ H^{n},\Theta^{1} H^{n-1}, \Theta^{2} H^{n-2},\ldots, \Theta^{n-1} H^{1}, \Theta^{n} \},$ our goal is to show  
\begin{equation}\label{eq:ksum}
\sum_{i=0}^n p_i\Theta^{i}H^{n-i}=0\quad\Rightarrow \quad p_i=0,
\end{equation}
where $p_i\in R$.
Let $b\in \mathbf{b}$ be a bond, and let's apply $(g_0^{b})^*$ to the sum. For each term in \eqref{eq:ksum} we obtain:
\begin{equation}\label{eq:cases}
(g_0^{b})^*(p_i\Theta^iH^{n-i})
\begin{cases}
    p_i\Theta^{i-1}H^{n-i}\unknot, & \text{if } b \text{ is a bond in } \Theta,\\
    p_i(-A^2-A^{-2})\Theta^i H^{n-i-1}\unknot, & \text{if } b \text{ is a bond in } H.\\
\end{cases}    
\end{equation}
There are four cases to study, depending on the position of $b$.

\medskip

\noindent {\it Case 1.} If $b$ is a bond in $\Theta$, and there exists no such $j\in \{1,2,\ldots,n\}$ that $(g_0^{b})^*(\Theta^{i}H^{n-i})=(g_0^{b})^*(\Theta^j H^{n-j})$, then, by the induction hypothesis, it follows that $p_i=0$.

\medskip

\noindent {\it Case 2.} 
Similarly, if $b$ is a bond in $H$, and there exists no such $j\in \{1,2,\ldots,n\}$ that $(g_0^{b})^*(\Theta^{i}H^{n-i}) = (g_0^{b})^*(\Theta^j H^{n-j})$, then, by the induction hypothesis, we have $p_i(-A^2 - A^{-2}) = 0$, and since $R$ is an integral domain, it follows that $p_i = 0$.

\medskip

\noindent {\it Case 3.} If $b$ is a bond in $\Theta$, and if there exists a $j \in \{1,2,\ldots,n\}$, such that $(g_0^{b})^*(\Theta^{i}H^{n-i})=q(g_0^{b})^*(\Theta^j H^{n-j})$ for some $q\in R$, then $j=i-1$ and $b$ is a bond in $H$ in the left-hand side term $\Theta^j H^{n-j}$. We have:
\begin{equation}\label{eq:case2}
(g_0^{b})^*(p_j\Theta^jH^{n-j}) = 
    p_j(-A^2-A^{-2})\Theta^j H^{n-j-1}\unknot.
\end{equation}
By the induction hypothesis it follows from \eqref{eq:cases} and \eqref{eq:case2} that
 \begin{equation}\label{eq:ih1}
    p_i + p_j\Kauff = 0.
\end{equation}
We apply also $(g_\infty^{b})^*$ to the terms $p_i\Theta^{i}H^{n-i}$ and $p_{j}\Theta^{j}H^{n-j}$, and we obtain:
\begin{align*}
    (g_\infty^{b})^*(p_i\Theta^{i}H^{n-i}) &= p_i(-A^2-A^{-2})\Theta^{i-1}H^{n-i},\\  
    (g_\infty^{b})^*(p_j\Theta^{j}H^{n-j}) &= p_{j}\Theta^{j}H^{n-j-1}.
\end{align*}
Since $j=i-1$, it follows from the induction hypothesis that
\begin{equation}\label{eq:ih2}
    \Kauff p_i + p_j = 0.
\end{equation}
The system of equations given by \eqref{eq:ih1} and \eqref{eq:ih2} yields the solution
$$p_i(\Kauff^2-1) = 0 \qquad \text{and} \qquad  p_j(\Kauff^2-1) = 0.$$
Since $1-(-A^2-A^{-2})^2\neq 0$ and $R$ is an integral domain, it follows that $p_i = p_j =0$.

\medskip

\noindent {\it Case 4.} If $b$ is a bond in $H$, and if there exists a $j \in \{1,2,\ldots,n\}$, such that $(g_0^{b})^*(\Theta^{i}H^{n-i})=q(g_0^{b})^*(\Theta^j H^{n-j})$, for some $q\in R$, then $j=i+1$ and $b$ is a bond in $\Theta$ in the left-hand side term $\Theta^j H^{n-j}$. The equality $p_i = p_j =0$ follows from the same computations as in Case 3, except that we exchange $i$ and $j$.

\medskip

In summary, the framed bonded Kauffman bracket skein module $\kbsmfr$ is freely generated by $\basisfr$. \qed

\subsection{Normalization of the framed bonded KBSM}
Given a framed bonded knot or link $K$, we can evaluate the class $[K]$ in $\kbsmfr$ and express the class in terms of the basis. This expression, $[K]_\basisfr$, is by \autoref{thm:main} a framed bonded isotopy invariant of $K$.

If we choose an orientation on $K$, we can use the same techniques as in the classical bracket polynomial (definition \ref{def:bracket}) and defined a normalized rigid version $[K]_{\text{norm}}$ of $[K]$ in the following manner:
$$[K]_{\text{norm}} = (-A^3)^{-w(K)}[K]_{\basisfr},$$ where $w(K)$ is the writhe of the diagram $K$.

Furthermore, if $K$ is a framed bonded knot with $n$ bonds, and we wish to compute $[K]_\basisfr$ (or $[K]_\text{norm}$), we have to apply equation \eqref{eq:anti4} $(n-1)$-times to extract the $\Theta$ and $H$ generators. Thus, $[K]_\text{norm}$ is an expression of the form 
$$[K]_\text{norm} = \frac{1}{(1 + A^4)^{n-1}(1 + A^4 + A^8)^{n-1}} \big(p_0\,H^n + p_1\,\Theta H^{n-1} + p_2\,\Theta^2 H^{n-2} +\cdots + p_n\Theta^n\big),$$
 for some $p_i \in R.$ To avoid redundancy, we reduce this expression by multiplying $[K]_{\text{norm}}$ by $(1 + A^4)^{n-1}(1 + A^4 + A^8)^{n-1}$, and we define the \emph{reduced rigid bonded bracket polynomial} as
$$\llbracket K \rrbracket = (-A^3)^{-w(K)}(1 + A^4)^{n-1}(1 + A^4 + A^8)^{n-1} \, [K]_{\basisfr},$$
which is an element of the ring of Laurent polynomials in $A$ and polynomials in $\Theta$ and $H$ with integer coefficients, i.e. $\llbracket K \rrbracket \in \Z[A^{\pm 1}, \Theta, H]$. For an explicit computation of the invariant see \autoref{sec:example}.

\section{The case of topological bonded knots} \label{sec:topological}

Let us now consider the topological setting of bonded knots, where isotopy is generated by Reidemester moves presented in \autoref{fig:reidemeister}. Allowing move V, handcuff generators in this setting can be expressed  with thetas only, indeed:
$$
\raisebox{-0.9em}{\includegraphics[page=77]{pics/proofs.pdf}} = 
\raisebox{-0.9em}{\includegraphics[page=78]{pics/proofs.pdf}} = 
-A^{3}\,\raisebox{-0.9em}{\includegraphics[page=79]{pics/proofs.pdf}} =
-A^{3} \Big(A\, 
\raisebox{-0.9em}{\includegraphics[page=77]{pics/proofs.pdf}}
-A^{-1}\,
\raisebox{-0.9em}{\includegraphics[page=80]{pics/proofs.pdf}}
\Big)=
A^4\,
\raisebox{-0.9em}{\includegraphics[page=77]{pics/proofs.pdf}}
+A^{2}\,
\raisebox{-0.9em}{\includegraphics[page=81]{pics/proofs.pdf}}
$$
from where it follows that 
\begin{equation}\label{eq:h}
   H = (-A^2-A^{-2}) \Theta. 
\end{equation}

\begin{definition}
Let $R$ be an integral domain, where $A$ and $1+A^{4}$ are invertible in $R$.
Let $\linkstop$ a set of topological-vertex bonded links and let $R\linkstop$ the free $R$-module spanned by $\linkstop$. Let $\submodule'$ be the submodule of $R\linkstop$ generated by skein expressions 
\begin{equation*}
\icon{10} - A \icon{11} - A^{-1} \icon{12} \qquad \text{and} \qquad L \sqcup \nevozel - \Kauff L.
\end{equation*}
We define the topological bonded Kauffman bracket skein module as the quotient module $$\kbsmtop = R\linkstop \,/\,\submodule'.$$
\end{definition}

\begin{theorem}$\kbsmtop$ is freely generated by $\basistop = \{\emptyset, \Theta, \Theta^2, \ldots \}.$
\end{theorem}
\begin{proof}
As before, $\kbsmtop$ admits a natural grading
$$\kbsmtop = \bigoplus_{n=0}^\infty \kbsmtop_n.$$
To show that $\kbsmtop_n$ is generated by $\Theta^n$, we 
observe that equalities  \eqref{eq:anti1} and \eqref{eq:anti2}, as well as formula \eqref{eq:anti4} still holds in $\kbsmtop$. Using \eqref{eq:h}, the equality \eqref{eq:anti4} simplifies to
\begin{equation*} 
(-A^2-A^{-2})\Q{1}  =  \Theta \Q{21} \quad \text{and} \quad (-A^2-A^{-2})\Q{10}  =  \Theta \Q{20}
\end{equation*}
Thus, we can use the above equality to express a diagram with $n$ bonds as a term of the form $p\, \Theta^n$, where $p \in R.$

For freeness, we can use the same procedure as in \autoref{sec:free}, except that the computation greatly simplifies due to the fact that we have less generators and we only need to consider the simplified maps $$(g_0)^*: \Theta \rightarrow \unknot \qquad \text{and} \qquad (g_\infty)^*: \Theta \rightarrow (-A^2-A^{-2})\unknot.$$ 
\end{proof}

\section{An example} \label{sec:example}
Let us take the TRTX-Tp1a toxin protein from \autoref{fig:valvet} and denote the associated bonded knot by $K$. This protein is also interesting from the viewpoint that is contains a cystine knot \cite{cardoso2019structure, dabrowski2019knotprot}.
\begin{figure}[h]
    \centering
    \includegraphics[width=0.5\textwidth,page=3]{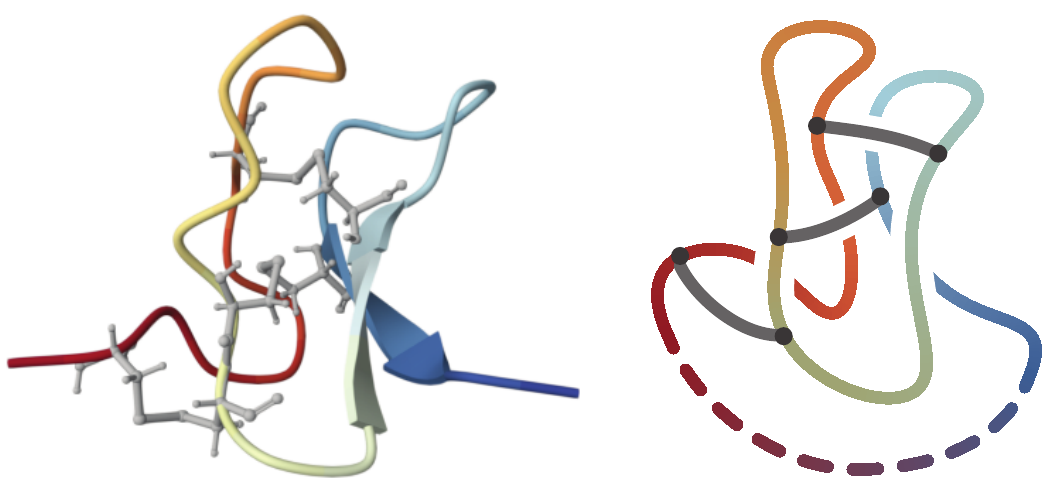} 
    \caption{{\bf The bonded structure of the TRTX-Tp1a toxin.} Left: the 3D ribbon model of the TRTX-Tp1a toxin (PDB 2MXM) from the Peruvian green velvet tarantula (Thrixopelma pruriens), a species rarely kept as a pet due to its tendency to react aggressively to minimal provocation. Right: the associated bonded knot diagram closed with a direct segment. Toxins frequently comprise short peptide chains stabilized by multiple disulfide bonds. These bonds are essential for maintaining the compact, bioactive conformation critical to the toxin’s potency at low concentrations. Additionally, the disulfide bonds enhance structural stability and resistance to degradation, ensuring that the toxin remains active longer.} 
    \label{fig:valvet}
\end{figure}

For easier computations, we introduce two new variables, $\delta = (-A^2-A^{-2})$ (the Kauffman term) and $\mu = (A^{-4}+1+A^4)$, and rewrite equation \eqref{eq:anti4} in a more compact form:
\begin{equation} \label{eq:ab}
\Q{1} = \alpha \Q{21}  + \beta \Q{20}
\end{equation}
where 
$\alpha= \frac{1}{\delta\mu} (\delta H - \Theta)$ and 
$\beta = \frac{1}{\delta\mu} (\delta \Theta - H)$.

Before computing the value of $[K]_\basisfr \in \kbsmfr$, we first compute the values of the following framed bonded knots, which will be used in subsequent computations (we omit the bracket for clarity):
\begin{align*}
\A{-0.3em}{12} & \eqeq{eq:sum} \tfrac{1}{\delta} H^2
\\
\A{-0.3em}{13} & \eqeq{eq:sum} \tfrac{1}{\delta}  \Theta H
\\
\A{-0.3em}{14} & \eqeq{eq:ab} \alpha \, H + \beta \, \Theta 
= \tfrac{1}{\mu}\Theta^2-\tfrac{2}{\delta\mu}H\Theta + \tfrac{1}{\mu}H^2
\\
\A{-0.3em}{15} & \eqeq{eq:sum} \tfrac{1}{\delta} \Theta^2  
\\
\A{-1.4em}{18} & = \A{-1.4em}{16} \eqeq{eq:kinkA} -A^{-3} \A{-1.4em}{17} 
\eqeq{eq:skein} -A^{-3} ( A \A{-1.4em}{19} +A^{-1} \A{-0.2em}{14} )
=
\tfrac{-1}{A^4\mu}\Theta^2
+\tfrac{2}{A^4\delta\mu}\Theta H
-\tfrac{A^2}{\delta\mu} H^2
\\
\A{-1.4em}{7} & \eqeq{eq:kinkA} \A{-1.4em}{8} =-A^{-3} \A{-1.4em}{9}
\eqeq{eq:skein} -A^{-3} ( A \A{-1.4em}{19} +A^{-1} \A{-1.4em}{18} ) \\
&\,= 
\tfrac{1}{A^8\mu} \Theta^2
-\tfrac{2}{A^8\delta\mu} \Theta H
-\tfrac{\left(1+A^{8}\right)}{A^6\delta\mu} H^2
\end{align*}
Similarly, we can compute
$$ \A{-1.4em}{6} = 
-\frac{1}{\delta\mu}\Theta^2
+\frac{A^{-6}+A^{6}}{\delta\mu}H\Theta
+\frac{A^{-4}+A^{4}}{\delta\mu}H^2.
$$

We now proceed to compute the invariant for our toxin. First, we use the skein relations \eqref{eq:skein} to remove crossings of $K \in \kbsmfr$, then we use equations \eqref{eq:sum} and \eqref{eq:ab} to extract one bond:
\begin{align*}
K = \A{-1.8em}{1} &= A^2 \A{-1.8em}{2}  + \A{-1.8em}{3} + \A{-1.8em}{4} + A^{-2}\A{-1.8em}{5} \\
&= \frac{A^2}{\delta} \Theta\, \A{-0.9em}{7}  +  \Theta\, \A{-0.9em}{7} + \Big(\alpha \A{-0.9em}{6}+ \beta \A{-0.9em}{7}\Big) +\frac{1}{A^{2}\delta} \Theta\, \A{-0.9em}{7}.
\end{align*}
Substituting the 2-bond knots with the previously computed ones, we obtain
\begin{align*}
 [K]_{\basisfr} =  \frac{1}{\delta^2\mu^2}\Big(&
(A^{-12}+2A^{-8}+A^{-4}+1)\Theta^3
+(3A^{-10}-A^6+2A^{-6}+A^2+A^{-2})\Theta^2H\\
&+(-A^8+2 A^{-8}-A^4-A^{-4}+1)\Theta H^2
+(-A^6-1A^{-2})H^3.
\Big)
\end{align*}    
The reduced rigid bonded bracket polynomial is
\begin{align*}
 \llbracket K \rrbracket =&\; 
(A^{12}+A^8+2A^4+1)\Theta^3
+(-A^{18}+A^{14}+A^{10}+2A^6+3A^2)\Theta^2H\\
&+(-A^{20}-A^{16}+A^{12}-A^8+2A^4)\Theta H^2
+(-A^{18}-A^{10})H^3.
\end{align*}    
By substituting $H = \delta \,\Theta$ \eqref{eq:h} and by $w(K)=0$, the topological Kauffman bracket skein module of $K$ is
$$[K]_{\basistop} = \frac{A^4}{(1+A^4)^2} \Theta^3.$$

\section*{Acknowledgments}
This work was financially supported by the Slovenian Research and Innovation Agency grants P1-J1-4031 and N1-0278, and program P1-029.
The authors declare that they have no conflict of interest.

\bibliographystyle{abbrv}
\bibliography{biblio}

\end{document}